\documentclass[11pt]{amsproc}
\usepackage[T1]{fontenc}
\usepackage[utf8]{inputenc}
\usepackage[a4paper,margin=1.1in]{geometry}
\usepackage{amssymb,amsmath,amsfonts,amsthm}
\usepackage{microtype}
\usepackage{graphicx}
\usepackage{booktabs}
\usepackage{longtable}
\usepackage{hyperref}

\newcommand{\cfgcell}[3]{%
\begin{minipage}[t]{0.237\textwidth}\centering
\includegraphics[width=\linewidth]{#1}\\[-30pt]
\scriptsize (#2)\ \ #3
\end{minipage}}

\newtheorem{theorem}{Theorem}

\newtheorem{definition}[theorem]{Definition}
\newtheorem{remark}[theorem]{Remark}

\title{On planar sections of the dodecahedron}
\author{Andreas Thom}
\address{Andreas Thom, TU Dresden, 01062 Dresden, Germany}
\email{andreas.thom@tu-dresden.de}

\begin{document}
\maketitle
\begin{abstract}
In the analysis of three-dimensional biological microstructures such as organoids, microscopy frequently yields two-dimensional optical sections without access to their orientation. Motivated by the question of whether such random planar sections determine the underlying three-dimensional structure, we investigate a discrete analogue in which the ambient structure is the vertex set of a Platonic solid and the observed data are congruence classes of planar intersections. For the regular dodecahedron with vertex set $V$, we define the planar statistic of a subset $X\subseteq V$ of vertices as the distribution of isometry types of inclusions $\Pi\cap X \subseteq \Pi \cap V \subseteq V$, and ask whether this statistic determines $X \subset V$ up to isometry. We show that this is not the case: there exist two non-isometric $7$-element subsets with identical planar statistics. 

As a consequence, there exist two polytopes in $\mathbb R^3$, whose distribution of isometry classes of two-dimensional intersections is identical, while the polytopes are not themselves isometric. This result is an analogue of classical non-uniqueness phenomena in geometric tomography.\end{abstract}

\section{Introduction}
High-resolution imaging techniques in biological systems frequently produce planar sections of inherently three-dimensional structures. A prominent example appears in organoid research, where multicellular aggregates are cultivated in three-dimensional media, yet standard microscopy visualizes them in two-dimensional optical slices, \cite{KesharaKimGrapinBotton22, KokEtAl25, OngEtAl25}. In many experimental pipelines, one does not obtain a calibrated full-volume scan, but rather a large collection of optical sections or histological slices taken at (approximately) random offsets and orientations.


From a mathematical viewpoint, such data amount to a distribution of planar configurations, observed only up to planar congruence. This leads to a basic inverse question: does the distribution of random planar sections identify the underlying three-dimensional structure uniquely? The problem is conceptually different from classical geometric tomography, where the position of the slicing plane is part of the measured data; here we treat the slice as unlabelled, and only the isometry class of the resulting planar configuration is retained.

The classical entry point in stereology is Wicksell's corpuscle problem for spheres \cite{Wicksell25}. For a single sphere of radius $r$, a plane at offset $u\in[-r,r]$ cuts a disk of radius
\[
\rho = \sqrt{r^2-u^2}\in[0,r].
\]
If the plane offset is sampled uniformly, the induced (non-normalized) section-radius measure on $(0,r)$ is
\[
\mu_r(d\rho)=\frac{2\rho}{\sqrt{r^2-\rho^2}}\,d\rho,\qquad \mu_r([0,r])=2r.
\]
Passing from a single sphere to a population, let $f(r)$ be the density of sphere radii in $\mathbb{R}^3$ and let $g(\rho)$ be the density of observed section radii. Then $g$ is related to $f$ by an Abel-type integral transform \cite{Wicksell25, JakemanAnderssen75}:
\[
g(\rho)=2\rho\int_{r=\rho}^{\infty}\frac{f(r)}{\sqrt{r^2-\rho^2}}\,dr,\qquad \rho>0,
\]
and one can invert this relation and recover $f$ from $g$ via
\[
f(r)= -\frac{1}{\pi}\frac{d}{dr}\int_{\rho=r}^{\infty}\frac{g(\rho)}{\sqrt{\rho^2-r^2}}\,d\rho,\qquad r>0.
\]
In particular, in the spherical case the radius distribution is uniquely determined by the section statistics.

Beyond spheres, identifiability becomes subtle. For particles of the form $sK$ (random orientation and scale $s$), section functionals such as area lead to one-dimensional integral equations, and the scale distribution can often be identified \cite{Santalo55, vanDerJagtJongbloedVittorietti2024}. Related stereological methods and estimators also exist for polyhedral shapes such as cubes \cite{OhserMucklich95, OhserNippe97}. On the other hand, there are genuine non-identifiability results: Cruz-Orive, solving a conjecture of Moran, exhibited distinct distributions of ellipsoids with identical induced distributions of unoriented planar sections \cite{CruzOrive76, Moran72}.

The aim of this work is to study such identifiability questions at first in a highly controlled model, motivated by random-slice data in microscopy. We fix a finite vertex set $V$ of a Platonic solid in $\mathbb{R}^3$ and regard subsets $X\subset V$ as discrete proxies for three-dimensional microstructures. Each vertex-plane $\Pi$ determines an inclusion $\Pi\cap X \subseteq \Pi\cap V \subseteq V$, and the only observable information is the Euclidean congruence type of this inclusion. Aggregating over all vertex-planes, we obtain a distribution of planar congruence classes; we call it the \emph{planar statistic} of $X$, see Definition \ref{def:planar-statistics}.

Our guiding question is whether the planar statistic determines $X$ up to isometry in $\mathbb{R}^3$. In analogy with non-uniqueness phenomena for intercept and chord-length distributions in the plane \cite{MallowsClark70}, one expects that low-dimensional section statistics can fail to be injective. Our main theorem confirms this expectation in the case of the regular dodecahedron: we exhibit two non-congruent $7$-element subsets $S,T\subset V$ with identical planar statistics, see Theorem \ref{thm:nonuniq}. In particular, even in an idealized, noise-free regime with isotropic plane sampling, uniformly sampled slices need not suffice to distinguish the underlying structure.

As a geometric consequence, we construct two polytopes in $\mathbb{R}^3$ whose distributions of isometry classes of unoriented planar sections coincide, although the bodies themselves are not isometric; see Theorem \ref{thm:convex-bodies}. So, while the result of Cruz-Orive \cite{CruzOrive76} shows non-identifiability for distributions of convex sets, our construction yields non-identifiability already for single polytopes.
This result is an analogue of classical non-uniqueness phenomena in geometric tomography \cite{Gardner06, Santalo04}. 

\medskip

Let us explain our approach and the relation to previous work in more detail.
Our strategy is closely analogous in spirit to the classical example of Mallows and
Clark \cite{MallowsClark70,MallowsClark71}, who exhibited two convex dodecagons in the
plane with identical chord–length distributions but which are not congruent. Their
construction naturally splits into a combinatorial step and a geometric encoding step.

In the combinatorial step, one starts from the regular octagon and considers its eight
sides or, equivalently, the eight isosceles sectors with apex at the centre. Mallows and
Clark choose two subsets \(X\) and \(Y\) of four sides each with the property that the
multiset of unordered pairs of sides in \(X\) is, up to Euclidean isometry of the octagon,
the same as the multiset of unordered pairs of sides in \(Y\). In modern language, \(X\)
and \(Y\) form a non-congruent homometric pair in the cyclic graph \(C_{8}\) that records
side adjacency. At this level the construction is purely discrete: one arranges that the
“incidence statistics’’ of \(X\) and \(Y\) inside the octagon agree, even though the subsets
themselves are not related by a symmetry.

The geometric step then promotes this combinatorial configuration to a pair of convex
subsets of the plane. To each side in \(X\), respectively in \(Y\), one attaches a small congruent
isosceles triangular hat pointing outward, thereby turning the octagon into two convex
dodecagons \(K_{X}\) and \(K_{Y}\). The contribution of a cap to the
chord–length distribution depends only on which side it sits on, so the matching of
unordered pairs of sides in \(X\) and \(Y\) forces the chord–length distributions of
\(K_{X}\) and \(K_{Y}\) to coincide, although the polygons themselves are not congruent.

The present dodecahedral construction follows the same two-step pattern in a
three-dimensional setting. Here the regular octagon is replaced by the regular
dodecahedron \(D\) with vertex set \(V\), and the role of chords is taken over by planar
sections through vertex–planes. The first, combinatorial step is the search for two
subsets \(S,T \subset V\) with identical planar statistics in the sense of Definition \ref{def:planar-statistics}.
Equivalently, for every vertex–plane \(\Pi\) the inclusion \(\Pi \cap S \subset \subset \Pi \cap V \subset V\) has, up to
isometry of \(D\), the same type and multiplicity distribution as \(\Pi \cap T \subset \Pi \cap V \subset V\).
Theorem \ref{thm:nonuniq} shows that such a non-congruent pair \((S,T)\) exists inside the vertex–plane incidence structure of the dodecahedron.

In the second step we again encode this discrete data into convex bodies by a local
modification at selected vertices. Starting from \(D\), we cut off a small congruent
triangular pyramid at each vertex \(v \in X\), obtaining a polytope \(K_{X}\) for
\(X \subseteq V\). When an affine plane \(H\) meets \(D\) near some vertices, the local
shape of the section \(K_{X} \cap H\) in a neighbourhood of those vertices is completely
determined, up to isometry, by the configuration \(\Pi \cap X \subset \Pi \cap V \subset V\) , where \(\Pi\) is the nearby
vertex–plane through the intersected vertices. The distribution of isometry classes of
planar sections of \(K_{X}\) can therefore be written as a convex combination of
contributions coming from the finitely many planar intersection types
\(\Pi \cap X \subset \Pi \cap V \subseteq V\), with coefficients prescribed precisely by the planar statistic
\(\mathrm{PS}(X)\).

From this viewpoint, Theorem \ref{thm:convex-bodies} is (upon replacing faces by vertices)the exact three-dimensional analogue of the Mallows–Clark theorem: the equality \(\mathrm{PS}(S) = \mathrm{PS}(T)\) ensures that the
section distributions of \(K_{S}\) and \(K_{T}\) coincide, while the underlying subsets
\(S\) and \(T\) and hence the patterns of truncated vertices of \(D\) are not related by
any isometry.

\section{Main results and model}
Let $V$ be the 20 vertices of a regular dodecahedron in the golden-ratio model. Let's write that out in full detail for convenience of the reader. Here, we set $\varphi = \frac{1 + \sqrt{5}}2.$
\begin{center}
\begingroup\small
\setlength{\tabcolsep}{6pt}
\begin{tabular}{r c  r c  r c  r c}
\toprule
idx & $(x,y,z)$ & idx & $(x,y,z)$ & idx & $(x,y,z)$ & idx & $(x,y,z)$ \\
\midrule
0 & $(-1,-1,-1)$ & 1 & $(-1,-1,1)$ & 2 & $(-1,1,-1)$ & 3 & $(-1,1,1)$ \\
4 & $(1,-1,-1)$ & 5 & $(1,-1,1)$ & 6 & $(1,1,-1)$ & 7 & $(1,1,1)$ \\
8 & $(0,\varphi^{-1},\varphi)$ & 9 & $(0,\varphi^{-1},-\varphi)$ & 10 & $(0,-\varphi^{-1},\varphi)$ & 11 & $(0,-\varphi^{-1},-\varphi)$ \\
12 & $(\varphi^{-1},\varphi,0)$ & 13 & $(-\varphi^{-1},\varphi,0)$ & 14 & $(\varphi^{-1},-\varphi,0)$ & 15 & $(-\varphi^{-1},-\varphi,0)$ \\
16 & $(\varphi,0,\varphi^{-1})$ & 17 & $(\varphi,0,-\varphi^{-1})$ & 18 & $(-\varphi,0,\varphi^{-1})$ & 19 & $(-\varphi,0,-\varphi^{-1})$ \\
\bottomrule
\end{tabular}
\endgroup

\end{center}
We depict those vertices in a Schlegel diagram as usual. The full set is depicted as follows:
\begin{figure}[h]\centering
\includegraphics[width=0.45\textwidth]{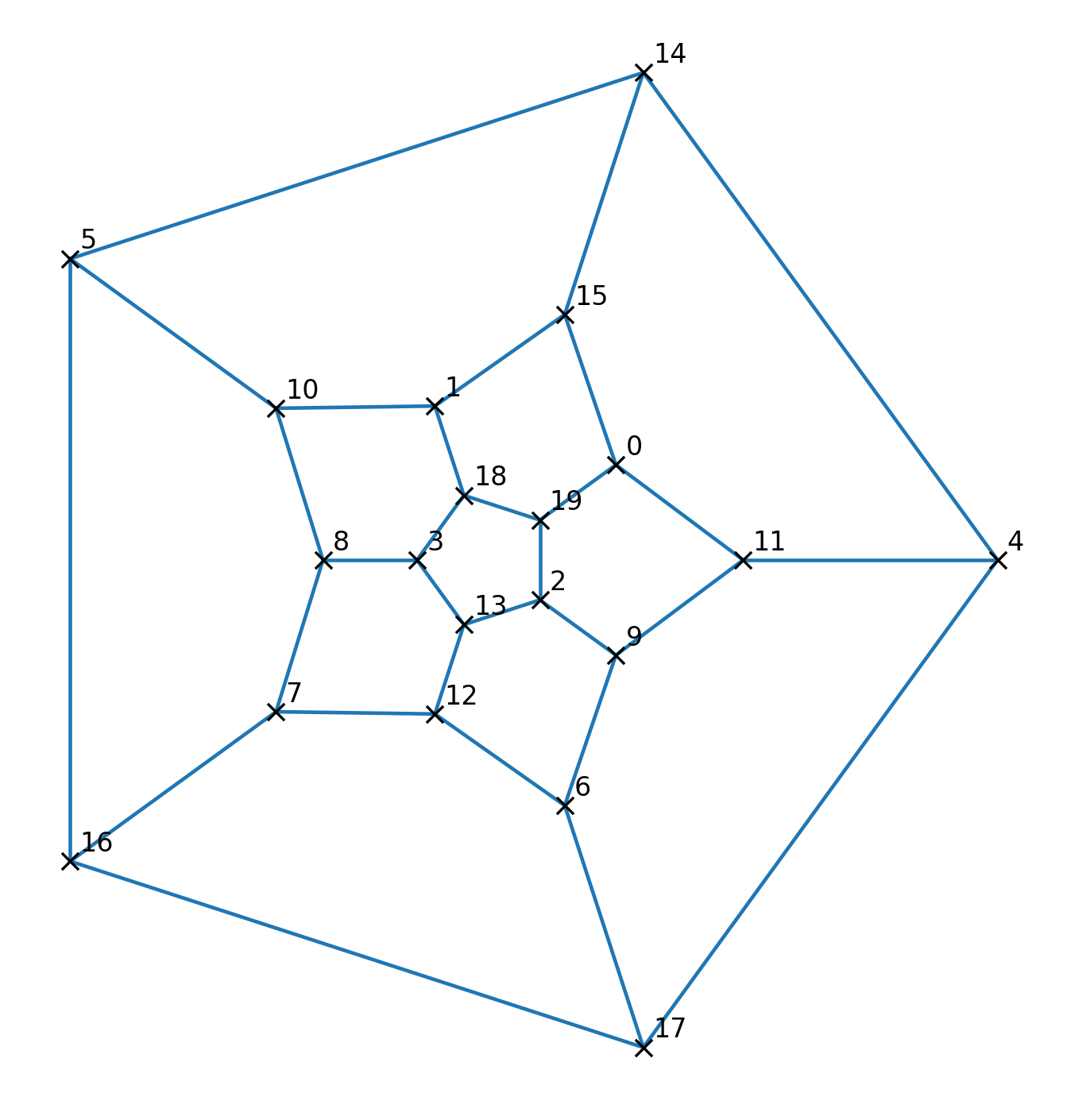}
\caption{Schlegel diagram with all 20 vertices labeled.}
\end{figure}

For $X\subset V$ the \emph{planar statistic} $\mathrm{PS}(X)$ is the multiset, over all planes $\Pi$ with $|\Pi\cap V|\ge 3$, of planar congruence classes of subset inclusions $\Pi\cap X \subseteq V$, seen as point sets in $\mathbb R^3$. More precisely,  let $D\subset\mathbb{R}^3$ be the dodecahedron as above with vertex set
$V$. A \emph{vertex-plane} is an affine plane
$\Pi\subset\mathbb{R}^3$ with $|\Pi\cap V|\ge 3$, and we denote by
\[
\mathcal{H} := \{\Pi\subset\mathbb{R}^3 : \Pi \text{ is a vertex-plane of } D\}
\]
the finite set of all vertex-planes of $D$.

For a subset $X\subset V$ and a vertex-plane $\Pi\in\mathcal{H}$ we consider the
inclusion of finite subsets $\Pi\cap X \subseteq V$.
Two vertex-planes $\Pi,\Pi'\in\mathcal{H}$ are called
\emph{planarly congruent with respect to $X$} if there exists a Euclidean
isometry $g\in\mathrm{Isom}(\mathbb{R}^3)$ such that $g(V) = V$, $g(\pi \cap V)=\Pi' \cap V$ and $g(\Pi\cap X) = \Pi'\cap X$.
We denote the corresponding equivalence relation on $\mathcal{H}$ by
$\Pi\sim_X \Pi'$ and write $[\Pi]_X$ for the equivalence class of
$\Pi$.
\begin{definition}\label{def:planar-statistics}
The \emph{planar statistic} of $X\subset V$ is the
multiset $\mathrm{PS}(X)$
of planar congruence classes $[\Pi]_X$ as $\Pi$ ranges over all vertex-planes
in $\mathcal{H}$, counted with multiplicity. 
\end{definition}
In other words,
$\mathrm{PS}(X)$ records, for each vertex-plane $\Pi$, the isometry type in
$\mathbb{R}^3$ of the inclusion $\Pi\cap X\subseteq \Pi \cap V \subseteq V$, and forgets the actual
position of $\Pi$.

\begin{remark} \label{rem:stat}
There is a slight variation of this definition which would record just the isometry type of the inclusion $\Pi \cap X \subseteq \Pi \cap V$ instead of $\Pi \cap X \subseteq V$. This is almost the same amount of information, except that in the case of a square $Q=\Pi \cap V$ with $\Pi \cap X$ two adjacent vertices. In this case, the two distinct embeddings of $Q$ into $V$ give rise to two different planar congruence
classes $[\Pi]_X$.
\end{remark}

It is a basic observation that the planar statistic of an image of $X$ under an isometry of $V$, orientation preserving or not, is the same as the one for $X$; i.e.\ isometric subsets of $V$ yield the same planar statistics. In view of the applications in shape recognition described in the introduction, our initial hope was that the planar statistics $\mathrm{PS}(X)$ allows one to recover $X$, but we show that this is not the case. Our main result is the following theorem:

\begin{theorem}\label{thm:nonuniq}
There exist two 7-element subsets $S,T\subset V$ that are not congruent in $\mathbb{R}^3$, yet $\mathrm{PS}(S)=\mathrm{PS}(T)$. One such pair is $S=\{0,1,2,3,4,11,17\}$ and $T=\{0,1,3,4,5,11,17\}$.
\end{theorem}

\begin{figure}[h]\centering
\includegraphics[width=0.38\textwidth]{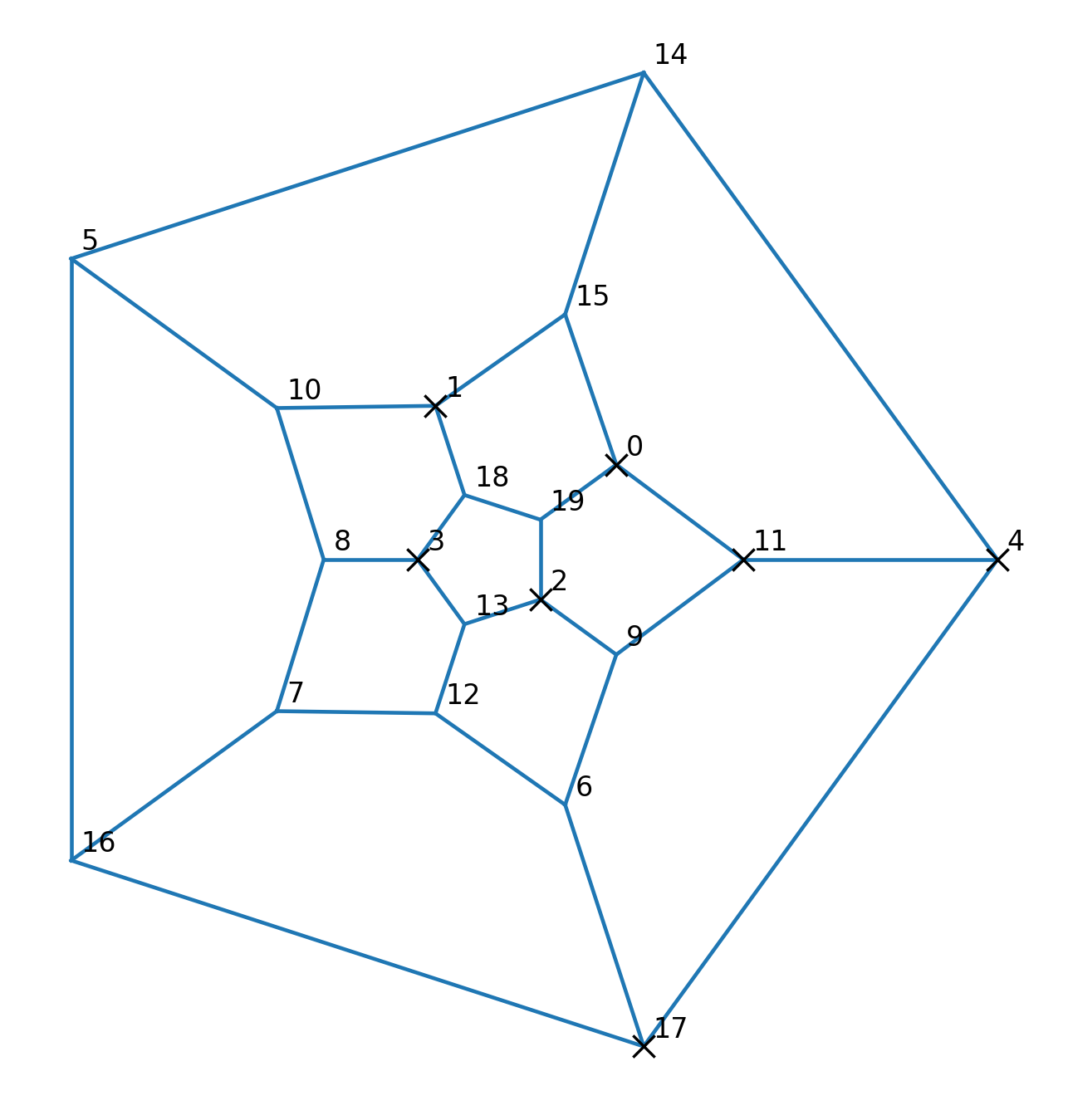}\hfill
\includegraphics[width=0.38\textwidth]{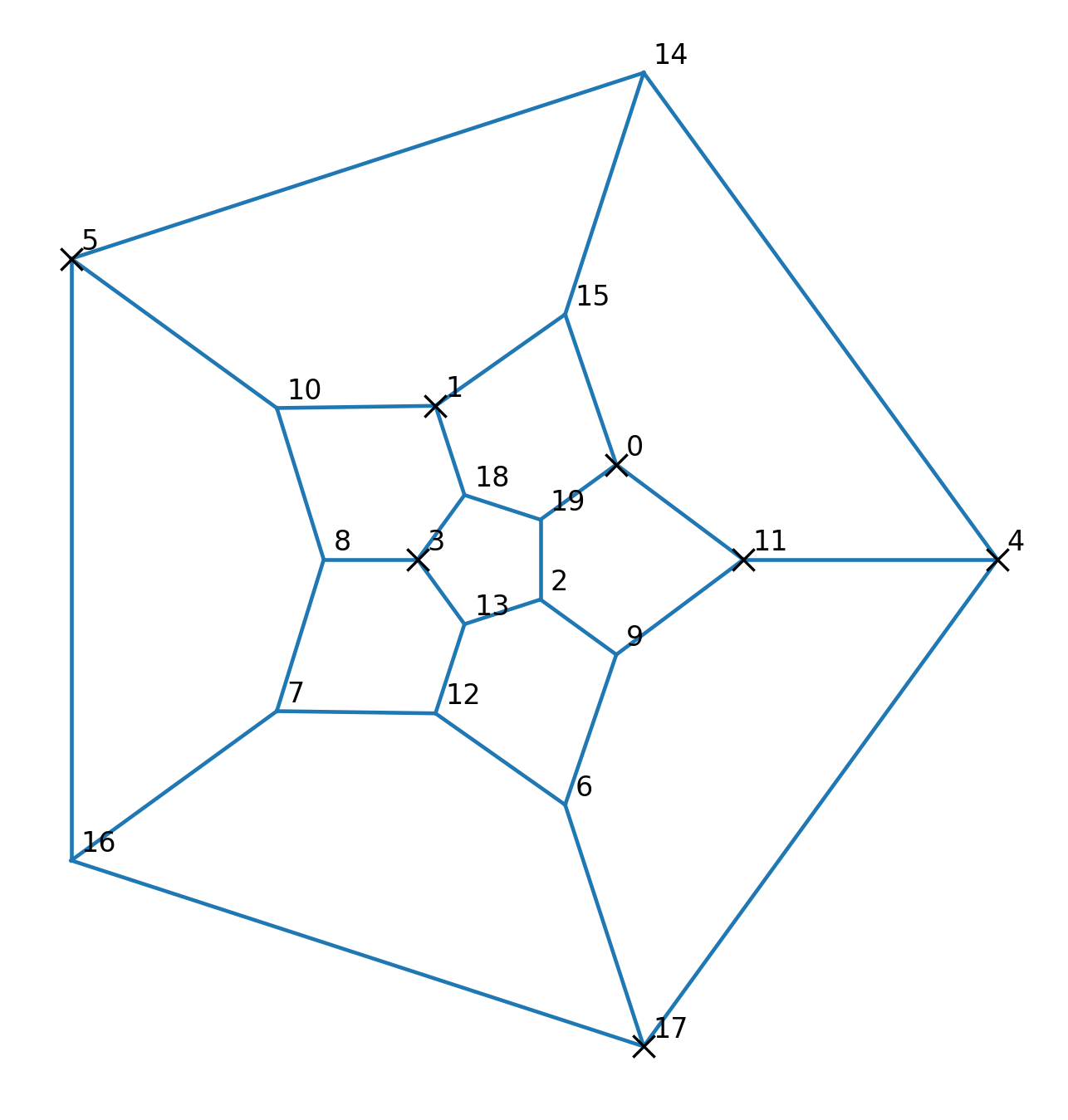}
\caption{7-element subsets $S$ (left) and $T$ (right)}
\end{figure}

In contrast, for all other Platonic solids, in particular for the icosahedron, we find injectivity of planar statistics up to congruence for all sizes of vertex sets -- a fact that can be verified computationally. The pair $(S,T)$ is the unique $7$-element pair of subsets up to isometry with the properties as stated in Theorem \ref{thm:nonuniq}, there are other pairs of size $r$ for $r \in \{8,\dots,13\}.$ 

\begin{proof}
The proof of Theorem \ref{thm:nonuniq} is in part computational: we enumerate all 319 vertex planes and aggregate subconfigurations by planar congruence. This shows that the planar statistics are identical, see below. Since $T=(S \setminus \{2\}) \cup \{5\}$, we only need to study the planes in $S$ that contain $2$ and compare to planes in $T$ that contain $5$; this can also be done by direct inspection with a three-dimensional model.

In order to see that $S$ and $T$ are not congruent, observe that each configuration contains a unique path (up to orientation) of length three; this is $(0,11,4,17)$ in both $S$ and  $T$. Now observe that the only endpoint of those paths, which does not have other neighbors at distance $\leq 2$ is the vertex $17$ in $S$. Hence, $S$ and $T$ cannot be congruent.
\end{proof}

Before we proceed, we need to say a bit more about the possible planar intersections with the vertex set of the dodecahedron. We group the 319 vertex planes by the isometry type of $\Pi\cap V$ and list their frequencies below. The picture shows convex hull of $\Pi \cap V$.

\begin{longtable}{@{}c c c c@{}}
\cfgcell{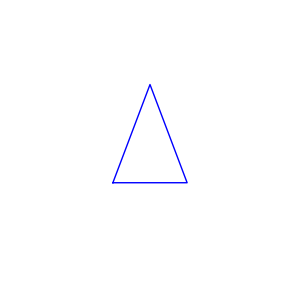}{3,0}{60} & \cfgcell{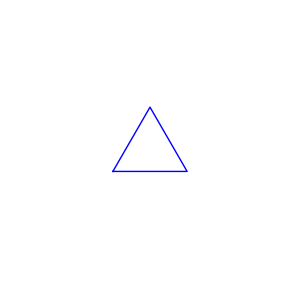}{3,0}{20} & \cfgcell{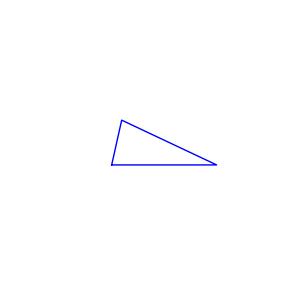}{3,0}{60} & \cfgcell{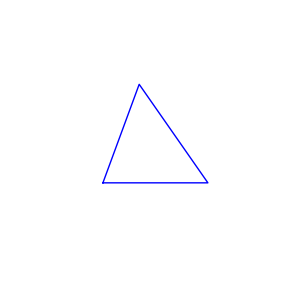}{3,0}{60} \\[-40pt]
\cfgcell{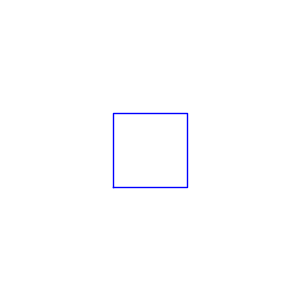}{4,0}{30} & \cfgcell{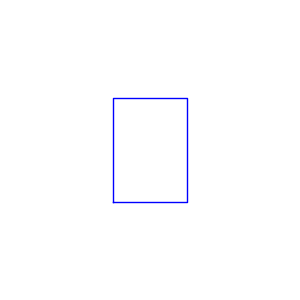}{4,0}{30} & \cfgcell{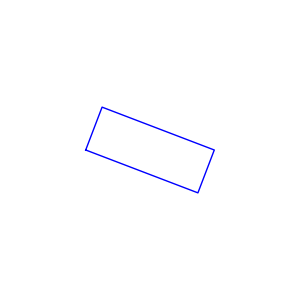}{4,0}{15} & \cfgcell{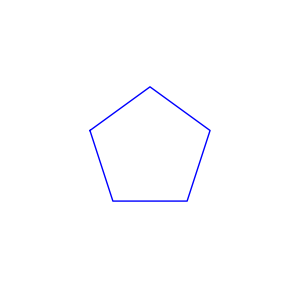}{5,0}{12} \\[-40pt]
\cfgcell{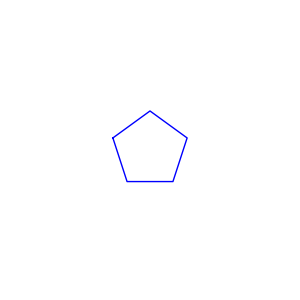}{5,0}{12} & \cfgcell{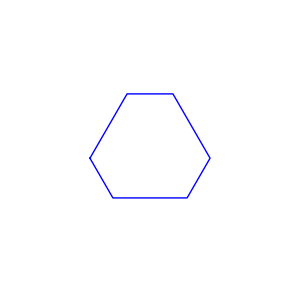}{6,0}{20} &  & \\
\caption{List of vertex planes and their frequencies}
\end{longtable}

This list fits with the count of 3-element subsets that we use to generate the planes, indeed we have:
$$(20 + 3\cdot 60)\binom{3}{3}+(15+2\cdot 30) \binom{4}{3}+2 \cdot 12 \binom{5}{3}+20\binom{6}{3}=1140 = \binom{20}3.$$
Subsequently we display the planar statistics for $S$ and $T$ aggregated by isometry type of the planar embedding $Z \subset V$. The main result follows then simply by comparison of the respective counts. Again, since $T=(S \setminus \{2\}) \cup \{5\}$, we only need to count planar configurations in $S$ containing $2$ and compare to planar configurations in $T$ containing $5$. In principle, this can be done by hand. The pictures in the following array show the convex hull of $\Pi \cap V$ with the subset marked. Below the picture one can read $(|\Pi \cap V|,|\Pi \cap S|)$ and the count of the configuration in $S$ and $T$.

\vspace{-0.8cm}

\begin{longtable}{@{}c c c c@{}}
\cfgcell{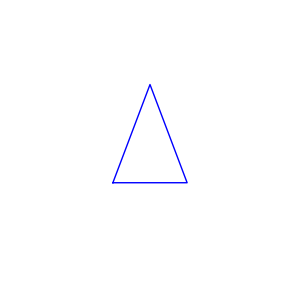}{3,0}{16} & \cfgcell{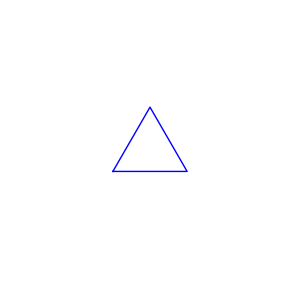}{3,0}{6} & \cfgcell{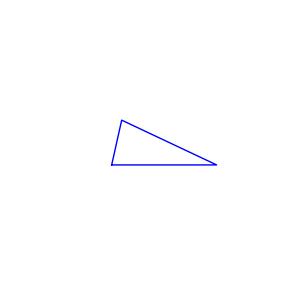}{3,0}{15} & \cfgcell{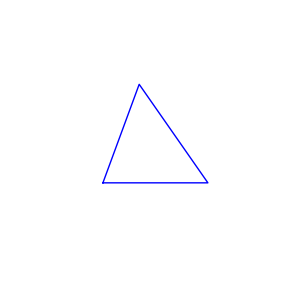}{3,0}{15} \\[-40pt]
\cfgcell{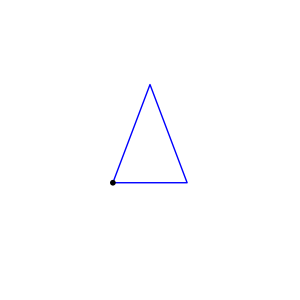}{3,1}{18} & \cfgcell{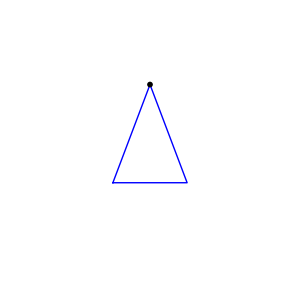}{3,1}{9} & \cfgcell{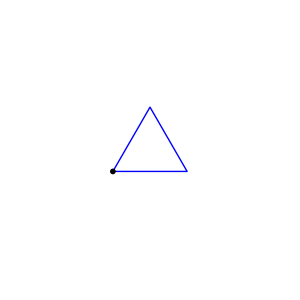}{3,1}{7} & \cfgcell{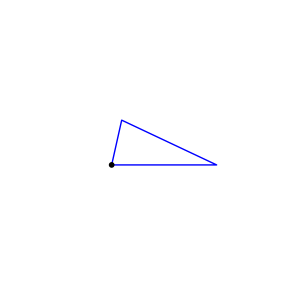}{3,1}{20} \\[-40pt]
\cfgcell{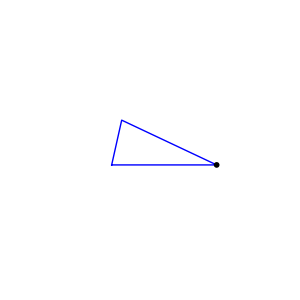}{3,1}{9} & \cfgcell{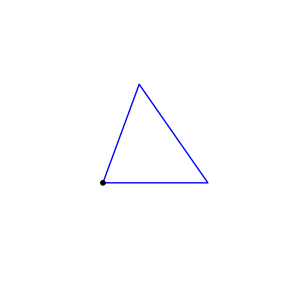}{3,1}{9} & \cfgcell{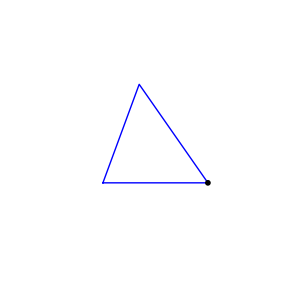}{3,1}{20} & \cfgcell{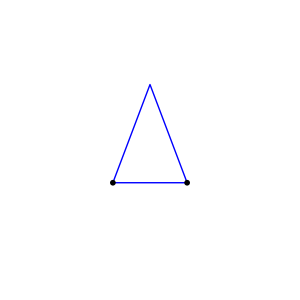}{3,2}{5} \\[-40pt]
\cfgcell{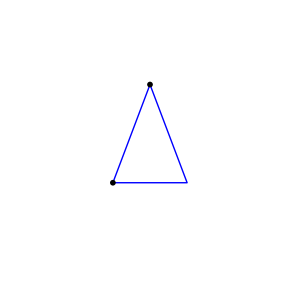}{3,2}{10} & \cfgcell{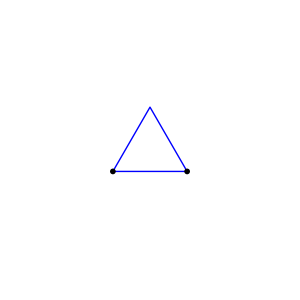}{3,2}{7} & \cfgcell{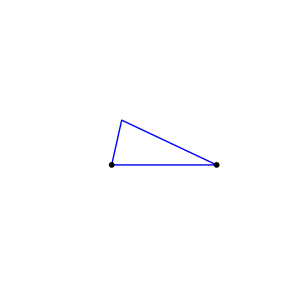}{3,2}{10} & \cfgcell{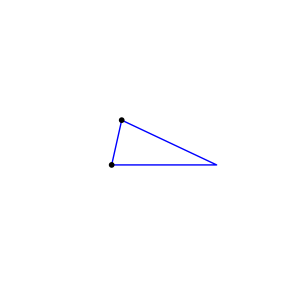}{3,2}{4} \\[-40pt]
\cfgcell{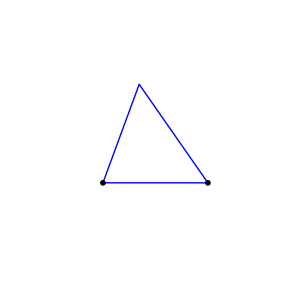}{3,2}{10} & \cfgcell{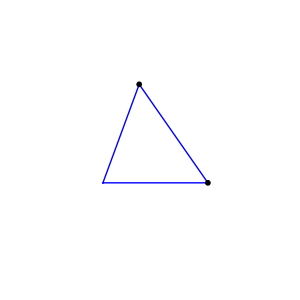}{3,2}{4} & \cfgcell{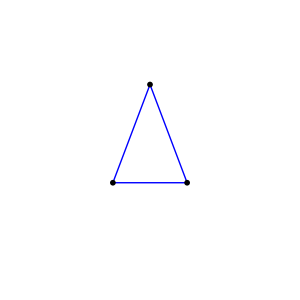}{3,3}{2} & \cfgcell{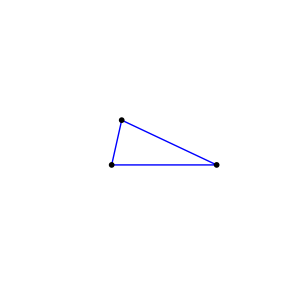}{3,3}{2} \\[-40pt]
\cfgcell{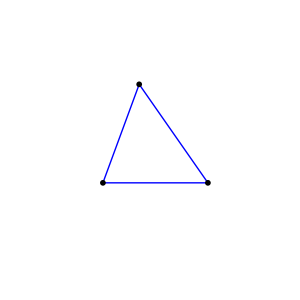}{3,3}{2} & \cfgcell{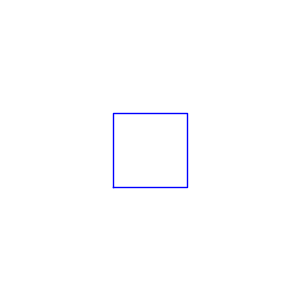}{4,0}{3} & \cfgcell{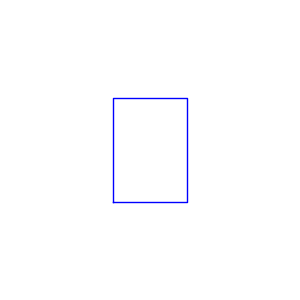}{4,0}{5} & \cfgcell{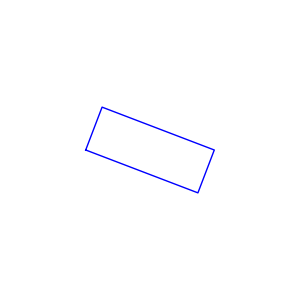}{4,0}{1} \\[-40pt]
\cfgcell{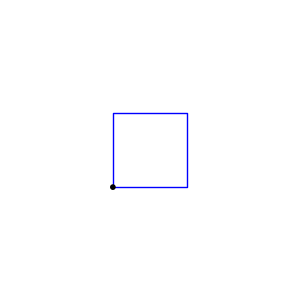}{4,1}{17} & \cfgcell{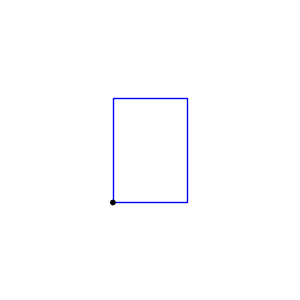}{4,1}{11} & \cfgcell{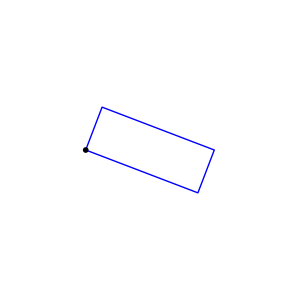}{4,1}{9} & \cfgcell{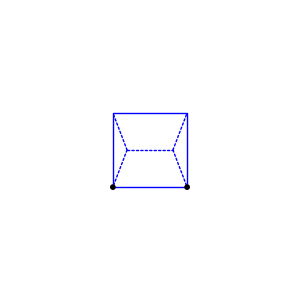}{4,2}{2} \\[-40pt]
\cfgcell{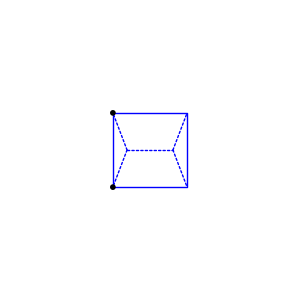}{4,2}{2} & \cfgcell{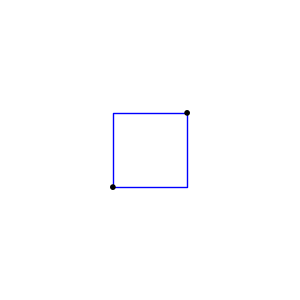}{4,2}{2} & \cfgcell{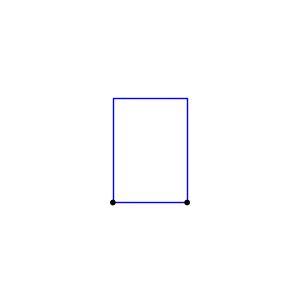}{4,2}{4} & \cfgcell{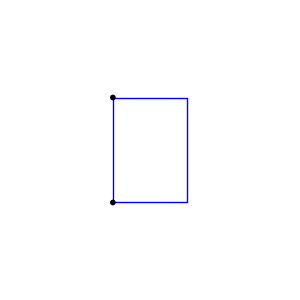}{4,2}{4} \\[-40pt]
\cfgcell{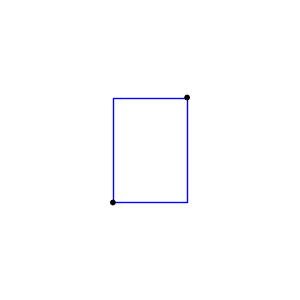}{4,2}{3} & \cfgcell{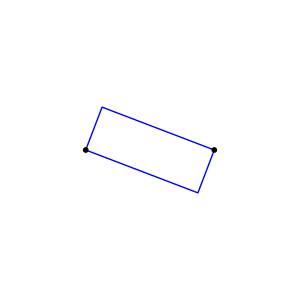}{4,2}{1} & \cfgcell{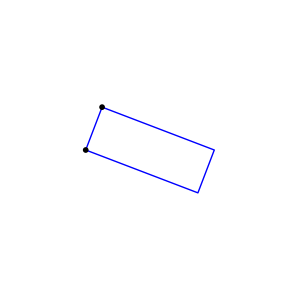}{4,2}{1} & \cfgcell{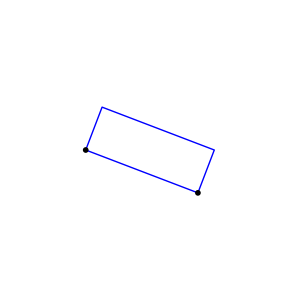}{4,2}{1} \\[-40pt]
\cfgcell{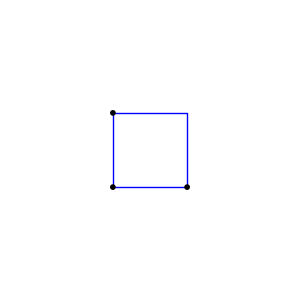}{4,3}{3} & \cfgcell{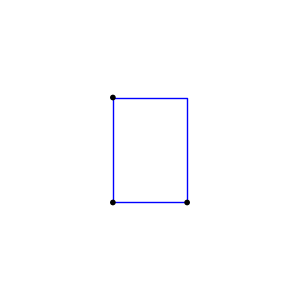}{4,3}{3} & \cfgcell{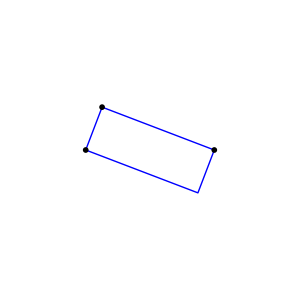}{4,3}{2} & \cfgcell{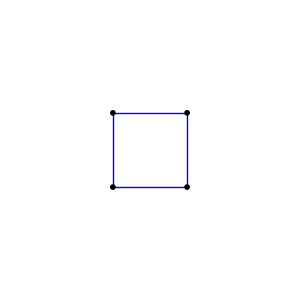}{4,4}{1} \\[-40pt]
\cfgcell{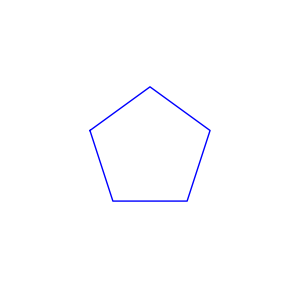}{5,0}{1} & \cfgcell{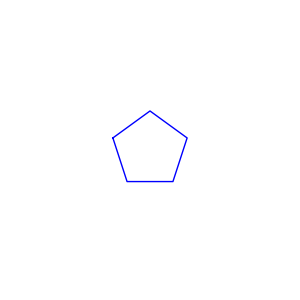}{5,0}{1} & \cfgcell{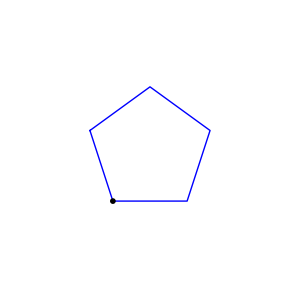}{5,1}{4} & \cfgcell{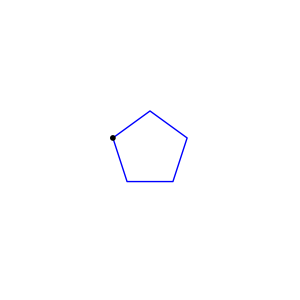}{5,1}{4} \\[-40pt]
\cfgcell{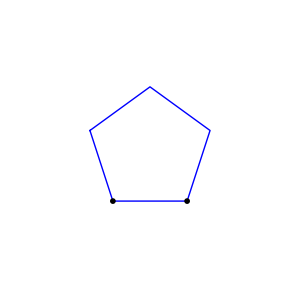}{5,2}{3} & \cfgcell{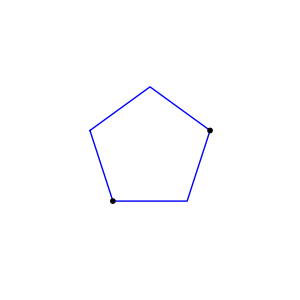}{5,2}{1} & \cfgcell{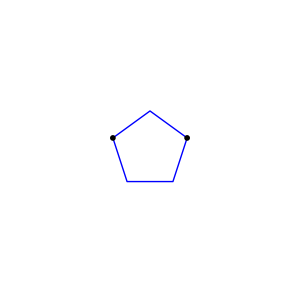}{5,2}{3} & \cfgcell{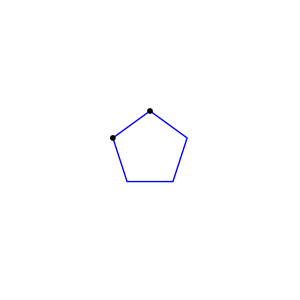}{5,2}{1} \\[-40pt]
\cfgcell{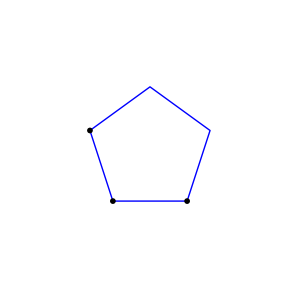}{5,3}{1} & \cfgcell{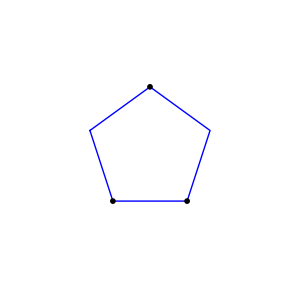}{5,3}{2} & \cfgcell{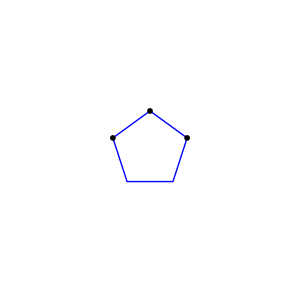}{5,3}{2} & \cfgcell{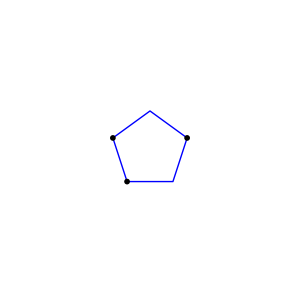}{5,3}{1} \\[-40pt]
\cfgcell{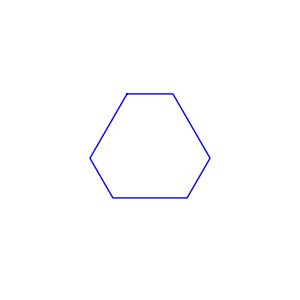}{6,0}{1} & \cfgcell{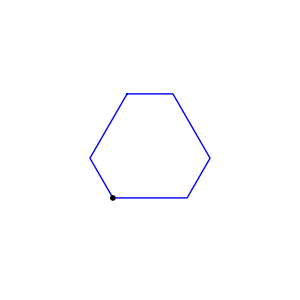}{6,1}{5} & \cfgcell{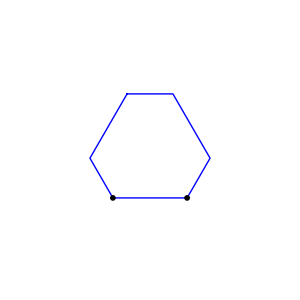}{6,2}{1} & \cfgcell{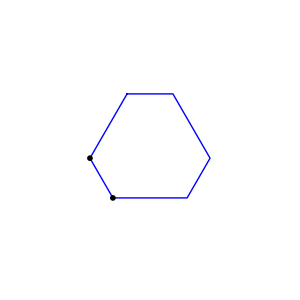}{6,2}{1} \\[-40pt]
\cfgcell{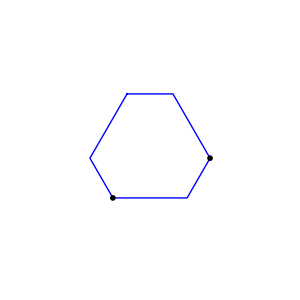}{6,2}{3} & \cfgcell{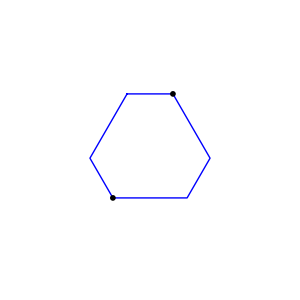}{6,2}{1} & \cfgcell{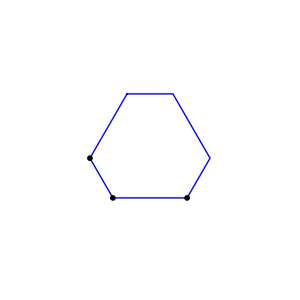}{6,3}{2} & \cfgcell{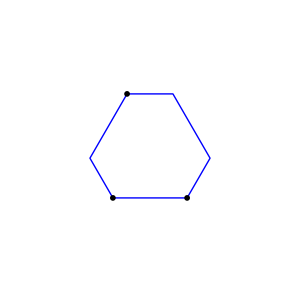}{6,3}{2} \\[-40pt]
\cfgcell{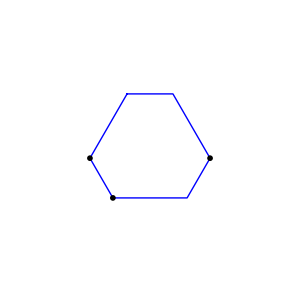}{6,3}{2} & \cfgcell{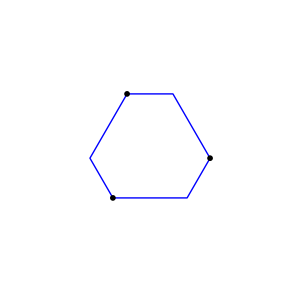}{6,3}{1} & \cfgcell{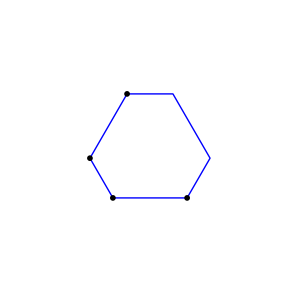}{6,4}{1} &  \\
\caption{List of subconfigurations with their frequency}
\end{longtable}

For the square with two adjacent vertices as a subset, there are actually two different isometric embeddings into $V$, see Remark \ref{rem:stat}; we indicated their difference by drawing the roof on the square as it appears in the dodecahedron.

\section{Convex bodies}

We will now refine the argument and show that $S,T$ can be used to construct three-dimensional polytopes $K_S$ and $K_T$ that are not isometric, but their distribution of isometry classes of planar intersections is identitical. Let us first explain more precisely, what that means:
We denote by $\mathcal{P}$ the space of compact planar convex sets modulo Euclidean isometries. For each closed convex set $K \subset \mathbb R^3$, we consider the natural map
\[
\Phi_K \colon \mathcal{A}(3,2)\longrightarrow\mathcal{P},\qquad
H\longmapsto [K\cap H]_{\mathrm{iso}},
\]
where $\mathcal{A}(3,2)$ is the affine Grassmannian of
planes in $\mathbb{R}^3$. Note that $\mathcal{A}(3,2)$ is naturally a fibre bundle over $\mathbb R{\rm P}(2)$ with fibre $\mathbb R$, where the fibre over a line $u\in\mathbb R {\rm P}(2)$ is the set of all planes with normal $u$, where the element in the fibre determines the offset. A natural measure $\lambda$, invariant under isometries of $\mathbb R^3$, can then be constructed as the product of the surface area measure on $\mathbb R{\rm P}(2)$ and the Lebesgue measure on the fibre $\mathbb R$. With this normalization of $\lambda$, we obtain $(\Phi_{K})_*(\lambda)(\mathcal P \setminus \{\varnothing \}) = 6 \cdot W_2(K)$ for a compact convex set $K$, where $W_2$ is the second quermassintegral. We call the Borel measure $(\Phi_K)_*(\lambda)$ on $\mathcal{P}$ the distribution of isometry classes of planar sections. 

\begin{theorem}\label{thm:convex-bodies}
There exist three-dimensional polytopes $K,L \subseteq \mathbb R^3$, such that
\begin{enumerate}
\item $K$ and $L$ are not congruent, but
\item $(\Phi_K)_*(\lambda) = (\Phi_L)_*(\lambda)$.
\end{enumerate}
\end{theorem}

\begin{proof}
Let $D\subset\mathbb{R}^3$ be a regular dodecahedron with vertex set $V$, and
let $S,T\subset V$ be two subsets as in Theorem~\ref{thm:nonuniq}: they
are not congruent in $\mathbb{R}^3$, but their planar statistics agree, i.e.\ 
$\mathrm{PS}(S) = \mathrm{PS}(T)$
as multisets of planar congruence classes in the sense of
Definition~\ref{def:planar-statistics}.

Let us now explain how we construct polytopes out of $S$ and $T$. For each $X\subset V$ define a convex polytope $K_X$ by truncating all vertices
$v\in X$ by small congruent triangular caps. We will show that $K_S$ and $K_T$ are not
congruent in $\mathbb{R}^3$, but the distributions of isometry classes of their
planar sections with respect to the invariant measure on affine planes coincide.

Fix once and for all a small parameter $\varepsilon>0$.
For each vertex $v\in V$ let $P_v$ be the unique plane that intersects the three
edges incident to $v$ in points of distance $\varepsilon$ from $v$, and denote
by $H_v^-$ the closed half-space bounded by $P_v$ that contains the barycenter
of $D$. For any subset $X\subset V$ we set
\[
K_X := D\cap\bigcap_{v\in X} H_v^-.
\]
By construction $K_X$ is obtained from $D$ by cutting off, at each $v\in X$, a
small congruent triangular pyramid. In particular, if $g\in \mathrm{Isom}(\mathbb{R}^3)$ is an
isometry with $g(K_S)=K_T$, then $g$ must map the set of truncated vertices of
$K_S$ onto that of $K_T$, and restricts to an isometry $g\colon D\to D$ with
$g(S)=T$. Since $S$ and $T$ are not congruent as subsets of $V$ by
Theorem~\ref{thm:nonuniq}, it follows that $K_S$ and $K_T$ are not
congruent.

We now compare the distributions of planar sections. For
$H\in\mathcal{A}(3,2)$ set
\[
m(H) := \bigl|\{v\in V : H\cap B(v,\varepsilon)\neq\emptyset\}\bigr|,
\]
the number of vertex neighbourhoods met by $H$. For each integer $m\ge 0$ define
\[
A_m := \{H\in\mathcal{A}(3,2) : m(H)=m\},
\]
so that $\mathcal{A}(3,2) = \sqcup_{m\ge 0} A_m.$
Let $p_m := \lambda(A_m)$ and $\lambda_m := p_m^{-1}\,\lambda\!\restriction_{A_m}$
for all $m$ with $p_m>0$. Then, we have
\[
\lambda = \sum_{m\ge 0} p_m\,\lambda_m,
\]
a canonical convex decomposition of $\lambda$, and for each $X\subset V$ we
obtain a corresponding convex decomposition
\[
\mu_X = \sum_{m\ge 0} p_m\,\mu_{X,m},
\qquad
\mu_{X,m} := (\Phi_{K_X})_*(\lambda_m).
\]
To prove $\mu_S=\mu_T$ it therefore suffices to show that $\mu_{S,m} = \mu_{T,m}$ for all $m\ge 0$.
\medskip

\emph{The cases $m=0$ and $1$.}
If $H\in A_0$, then $H$ misses all balls $B(v,\varepsilon)$ and in particular all caps.
The section $K_X\cap H$ then coincides with $D\cap H$; it does not depend on $X$. In particular, we can conclude $\mu_{S,0} = \mu_{T,0}.$
Similarly, if $H\in A_1$, then $H$ meets exactly one ball $B(v,\varepsilon)$, say at $v$. The local
shape of the section $K_X\cap H$ in a neighbourhood of $B(v,\varepsilon)$ is completely
determined by the local geometry at $v$ and thus only depends on whether $v\in X$ or not. Thus, since $D$
is vertex-transitive and $|S|=|T|$, we obtain $\mu_{S,1} = \mu_{T,1}$.

\emph{The case $m\ge 3$.}
Suppose $H\in A_m$ with $m\ge 3$. By choice of $\varepsilon>0$, this implies that
$H$ lies in a small neighbourhood of a unique vertex plane
$\Pi\subset\mathbb{R}^3$ with
\[
\Pi\cap V = \{v_1,\dots,v_m\},\qquad m\ge 3,
\]
and that the vertices $\{v_1,\dots,v_m\}$ are precisely those whose balls
$B(v_i,\varepsilon)$ meet $H$. The local configuration of the section $K_X\cap H$ is then,
up to isometry, completely determined by the isometry type of the inclusion $\Pi\cap X \subset V$
and the position of $H$ in a sufficiently small neighbourhood of $\Pi$.

The isometry group $D$ acts on the set of such local configurations. For
a fixed $X \subset V$ the orbits of this action are parametrized precisely by
the planar congruence classes $[\Pi]_X$ appearing in the planar statistic
$\mathrm{PS}(X)$: two vertex planes $\Pi,\Pi'$ give rise to the same local
configuration type if and only if there is an isometry $g$ such that $g(V)=V$ and
$g(\Pi\cap X)=\Pi'\cap X$.

Fix a $G$-orbit of local configurations $\mathcal{O}$, i.e.\ an orbit of a planar inclusion $Z \subset V$ with $|Z|\ge 3$. Consider the set of planes $H$ whose local configuration belongs to $\mathcal{O}$. By
invariance of $\lambda$ and congruence of all caps and faces, the restriction of
$\lambda$ to this set has the same total mass and the same normalized image
measure under $\Phi_X$, independent of $X$; it depends only on $\mathcal{O}$.
Thus for each such orbit $\mathcal{O}$ there is a probability measure
$\nu_{\mathcal{O}}$ on $\mathcal{P}$ describing the contribution of one
occurrence of $\mathcal{O}$, and the component $\mu_{X,m}$ can be written as a
finite convex combination
\[
\mu_{X,m} = \sum_{\mathcal{O}} c_{\mathcal{O}}(X)\,\nu_{\mathcal{O}},
\]
where the sum runs over all orbits $\mathcal{O}$ with $m\ge 3$, and
$c_{\mathcal{O}}(X)$ is a nonnegative coefficient proportional to the number of
vertex planes whose local configuration with respect to $X$ lies in $\mathcal{O}$.
By definition, the multiset of classes $[\Pi]_X$ in $\mathrm{PS}(X)$ records
exactly these multiplicities. Thus the coefficients $c_{\mathcal{O}}(X)$ are
completely determined by $\mathrm{PS}(X)$.

Since $\mathrm{PS}(S) = \mathrm{PS}(T)$ as multisets of planar congruence
classes, it follows that $c_{\mathcal{O}}(S) = c_{\mathcal{O}}(T)$ for every
orbit $\mathcal{O}$, and hence $\mu_{S,m} = \mu_{T,m}$ for all $m \geq 3$.

\emph{The cases $m=2$.} If $H\in A_2$, then $H$ meets exactly two balls $B(v,\varepsilon)$, say at vertices
$v,w\in V$. The local shape of $K_X\cap H$ is then, up to isometry, determined by
the orbit type of the unordered pair $\{v,w\}$ under the isometry group of
$D$ together with the membership of $v,w$ in $X$. Again, there are only finitely many
such orbit types of unordered pairs in $V$, and for each such type all
corresponding pairs have the same local geometry. Each such pair belongs to various vertex planes with three or more vertices, so its contribution is again
encoded in the planar statistic $\mathrm{PS}(X)$; in particular, the counts of
pairs of each orbit type contained in $X$ are determined by $\mathrm{PS}(X)$.
Since $\mathrm{PS}(S)=\mathrm{PS}(T)$, the contributions for $m=2$ agree: $\mu_{S,2} = \mu_{T,2}$.

\medskip

Combining all cases, we have shown that $\mu_{S,m} = \mu_{T,m}$ for every $m\ge 0$. Since
\[
\mu_S = \sum_{m\ge 0} p_m\,\mu_{S,m}
\quad\text{and}\quad
\mu_T = \sum_{m\ge 0} p_m\,\mu_{T,m}
\]
with the same weights $p_m$, we conclude that $\mu_S=\mu_T$. This shows that the
distributions of isometry classes of planar sections of $K_S$ and $K_T$ with
respect to $\lambda$ coincide, as claimed.
\end{proof}


\section*{Acknowledgments}

We thank Vadim Alekseev and Leon Renkin for interesting discussions on this topic.

\vspace{0.2cm}

We used AI-generated scripted checks to classify vertex-plane types and to tally planar statistics,
and independently cross-checked key computations.
We also acknowledge assistance of a large language model for drafting and
refactoring code used to generate figures and tables; all mathematical claims and
computational results were verified independently by us.


\begin{thebibliography}{99}

\bibitem{CruzOrive76}
L.-M.~Cruz-Orive,
\newblock \emph{Particle size-shape distributions: The general spheroid problem. I. Mathematical model},
\newblock J.~Microsc., \textbf{107} (1976), no.~3, 235--253.

\bibitem{Gardner06}
R.~J.~Gardner, \emph{Geometric Tomography} (2nd ed.), Encyclopedia of Mathematics and its Applications, 
Cambridge University Press, 2006.

\bibitem{JakemanAnderssen75}
A.~J.~Jakeman and R.~S.~Anderssen,
\newblock \emph{Abel type integral equations in stereology: I. General discussion},
\newblock J.~Microsc., \textbf{105} (1975), no.~2, 121--133.

\bibitem{KesharaKimGrapinBotton22} R.~Keshara, Y.~H.~Kim and A.~Grapin-Botton, \emph{Organoid Imaging: Seeing Development and Function}, Annual Review of Cell and Developmental Biology, \textbf{38} (2022), 447--466.

\bibitem{KokEtAl25} R.~N.~U.~Kok, T.~van~Meel, S.~M.~van~Genderen, B.~D.~DeGrood, J.~H.~S.~McClymont and E.~Bakker, \emph{Label-free cell imaging and tracking in 3D organoids}, Cell Reports Physical Science, \textbf{6} (2025), no.~4, 102522.

\bibitem{MallowsClark70} C.~L.~Mallows and J.~M.~C.~Clark, \emph{Linear-Intercept Distributions Do Not Characterize Plane Sets}, Journal of Applied Probability, \textbf{7} (1970), no.~1, 240--244.

\bibitem{MallowsClark71} C.~L.~Mallows and J.~M.~C.~Clark, \emph{Corrections to ``Linear-Intercept Distributions Do Not Characterize Plane Sets''}, Journal of Applied Probability, \textbf{8} (1971), no.~1, 208--209.

\bibitem{Moran72}
P.~A.~P.~Moran,
\newblock \emph{The probabilistic basis of stereology},
\newblock Adv.~Appl.~Probab., Special Supplement \textbf{4} (1972), 69--91.

\bibitem{OhserMucklich95}
J.~Ohser and F.~M\"ucklich,
\newblock \emph{Stereology for some classes of polyhedrons},
\newblock Adv.~Appl.~Probab., \textbf{27} (1995), 384--396.

\bibitem{OhserNippe97}
J.~Ohser and M.~Nippe,
\newblock \emph{Stereology of cubic particles: various estimators for the size distribution},
\newblock J.~Microsc., \textbf{187} (1997), no.~1, 22--30.

\bibitem{OngEtAl25} Q.~C.~Ong, A.~G.~Soldevila-Muntadas, A.~Filippi\textendash Chiela, C.~G.~C.~H.~Roque and F.~A.~Hoffman, \emph{Digitalized organoids: integrated pipeline for high-speed 3D analysis of organoid structures using multilevel segmentation and cellular topology}, Nature Methods, \textbf{22} (2025), article no.~02685 (online first).

\bibitem{Santalo55}
L.~A.~Santal\'o,
\newblock \emph{Sobre la distribuci\'on de los tama\~nos de corp\'usculos contenidos en un cuerpo a partir de la distribuci\'on en sus secciones o proyecciones},
\newblock Trabajos de Estad\'istica, \textbf{6} (1955), 181--196.

\bibitem{Santalo04}
L.~A.~Santal\'o, \emph{Integral Geometry and Geometric Probability} (2nd ed.), Cambridge University Press, 2004.

\bibitem{vanDerJagtJongbloedVittorietti2024}
T.~van~der~Jagt, G.~Jongbloed and M.~Vittorietti,
\newblock \emph{Stereological determination of particle size distributions for similar convex bodies},
\newblock Electron.~J.~Stat., \textbf{18} (2024), no.~1, 742--774.

\bibitem{Wicksell25} S.~D.~Wicksell, \emph{The Corpuscle Problem. A Mathematical Study of a Biometric Problem}, Biometrika, \textbf{17} (1925), no.~1--2, 84--99.

\bibitem{Wicksell26}
S.~D.~Wicksell,
\newblock \emph{The Corpuscle Problem: Second Memoir. Case of ellipsoidal corpuscles},
\newblock Biometrika, \textbf{18} (1926), 151--172.


\end{thebibliography}
\end{document}